\RequirePackage{fix-cm}
\documentclass[smallextended]{svjour3}       % onecolumn (second format)
\smartqed  % flush right qed marks, e.g. at end of proof
\usepackage{graphicx}
\usepackage{amssymb}
\usepackage{color}
\usepackage{MnSymbol}
\usepackage{url}
\usepackage{lineno}

\newcommand{\rmi}{\mathrm{i}}
\newcommand{\rmd}{\mathrm{d}}
\newcommand{\rme}{\mathrm{e}}
\newcommand{\re}{\mathrm{Re}}
\newcommand{\im}{\mathrm{Im}}
\begin{document}

\title{Holomorphic representation of minimal surfaces in simply isotropic space}


\author{Luiz C. B. da Silva
}


\institute{Da Silva, L. C. B. \at
              Department of Physics of Complex Systems,\\               Weizmann Institute of Science, 
              Rehovot 7610001, Israel\\
              \email{luiz.da-silva@weizmann.ac.il}}
%\date{(Manuscript version: 18/jan/2021)}
\date{First published in [J. Geom. (2021)
112:35] by Springer Nature (fulltext available at the link \url{https://doi.org/10.1007/s00022-021-00598-z})} %Received: date / Accepted: date}

\maketitle
%\linenumbers

\begin{abstract}
It is known that minimal surfaces in Euclidean space can be represented in terms of holomorphic functions. For example, we have the well-known Weierstrass representation, where part of the holomorphic data is chosen to be the stereographic projection of the normal of the corresponding surface, and also the Bj\"orling representation, where it is prescribed a curve on the surface and the unit normal on this curve. In this work, we are interested in the holomorphic representation of minimal surfaces in simply isotropic space, a three-dimensional space equipped with a rank 2 metric of index zero. Since the isotropic metric is degenerate, a surface normal cannot be unequivocally defined based on metric properties only, which leads to distinct definitions of an isotropic normal. As a consequence, this may also lead to distinct forms of a Weierstrass and of a Bj\"orling representation. Here, we show how to represent simply isotropic minimal surfaces in accordance with the  choice of an isotropic surface normal.
\end{abstract}

\section{Introduction}

It is well-known that minimal surfaces in Euclidean space can be parameterized in terms of holomorphic functions, which gives the so-called Weierstrass representation \cite{Barbosa-Colares1986,Weierstrass}: $\mathbf{x}(z)=\re(\int^z\phi\,\rmd z)$, where $\phi$ is a holomorphic isotropic curve in $\mathbb{C}^3$ with no real periods, i.e., each $\phi_j$ is holomorphic, $\phi_1^2+\phi_2^2+\phi_3^2=0$, and each $\re\int^z\phi_j\rmd z$ is path-independent. In addition, if we define $\phi_1=\frac{F}{2}(1-G^2)$, $\phi_2=\frac{\rmi F}{2}(1+G^2)$, and $\phi_3=FG$, the holomorphic function $G$ can be seen as the stereographic projection of the Gauss map of the minimal immersion $\mathbf{x}(z)$. (The function $F$ can be associated with the differential of the third coordinate, i.e., the differential of the height data.) An alternative to the Weierstrass representation is the  Bj\"orling representation \cite{LopezMMJ2018,SchwarzCrelle1875}, which consists in the Cauchy problem for the mean curvature, i.e., one prescribes a curve $c(s)$ and the unit normal $\mathbf{n}(s)$ along $c$. The corresponding minimal surface is parameterized by $\mathbf{x}(z)=\re\int^z[c'(w)-\rmi\, \mathbf{n}(w)\times c'(w)]\rmd w$, where $c(w)$ and $\mathbf{n}(w)$ are analytic extensions of $\mathbf{n}$ and $c$. Alternatively, we may prescribe the initial curve together with the tangent plane. This later version of the Bj\"orling representation has been already discussed in simply isotropic space \cite{StrubeckerAM1954} which helps establishing simply isotropic analogs of two theorems due to Schwarz concerning lines and planes of symmetries of minimal surfaces \cite{SchwarzCrelle1875}. (See also chapter 12 of \cite{Sachs1990}.) Recently, Sato showed that simply isotropic minimal surfaces  also admit a Weierstrass representation given by \cite{SatoArXiv2018}: $\mathbf{x}(u, v) = \textrm{Re} \int^{z} (F, \rmi F, 2FG) \rmd w$, $z=u+\rmi v\in U\subseteq\mathbb{C}$. When expressed in their normal form, i.e., as a graph, isotropic minimal surfaces are the graph of harmonic functions and, therefore, we can write $\mathbf{x}(z)=(z,\re\,h)$, where $h$ is holomorphic. It is worth mentioning that Strubecker, for example in Ref. \cite{StrubeckerAMSUH1975}, p. 154 right after Eq. (23), refers to this representation as the isotropic analog of the Weierstrass representation.  Sato's approach in fact gives $\mathbf{x}(z)=(z,\re\,h)$ by setting $F(z)=1$ and choosing $G$ appropriately, but it is more general since we do not have to write the surface as a graph. 

{By adopting a Weierstrass representation instead of writing an isotropic minimal surface as a graph, a} natural question then arises: why should we use two holomorphic functions to represent isotropic minimal surfaces when we know that one is enough? Note that in Euclidean space we do have a reduction in the number of ``degrees of freedom", from three to two\footnote{Any conformal minimal immersion has harmonic coordinates and, consequently, each of the three coordinates can be locally seen as the real part of a holomorphic function on the plane.}. However, the Euclidean experience also teaches us that the advantage of a holomorphic representation lies in the ability to control key geometric information of minimal surfaces. The reduction in the number of holomorphic functions we need to represent minimal surfaces then follows as an extra. In the case of a Weierstrass representation, we control the (stereographic projection of the) unit normal. Then, we may also ask whether it is possible to interpret part of the holomorphic data of an isotropic minimal surface $M^2$ as the stereographic projection of its Gauss map. Observe, however, that in simply isotropic space there are more than one choice for a Gauss map, as recently emphasized in Ref. \cite{KelleciJMAA2021}. Therefore, to answer the previous question we must also specify what Gauss map we have in mind. Here, we are going to show that it is possible to find a Weierstrass representation for simply isotropic minimal surfaces such that part of their holomorphic data can be associated with the stereographic projection of either their parabolic normal or their minimal normal. (In this respect, we show that Sato's choice for the holomorphic representation is slightly related to the minimal normal. See Subsect. \ref{subsect::OtherChoicesWeierstRep} for a discussion of  his motivations.) We also discuss the Bj\"orling representation and show that there are three ways of doing that.

The remaining of this text is divided as follows. In Sect. 2, we present some background material on simply isotropic geometry. In Sect. 3, we present and study some properties of the stereographic projection in the simply isotropic space from a sphere of parabolic type to the plane. In Sect. 4, we present the main results of this work, namely the isotropic Weierstrass representation and its relation to the extrinsic  geometry of isotropic minimal surfaces. Finally, in Sect. 5, we discuss the Cauchy problem for the minimal surface equation, i.e., the Bj\"orling representation. The last section contains our concluding remarks. 

\section{Geometric background: the simply isotropic space}

The simply isotropic space $\mathbb{I}^3$ is an example of a Cayley-Klein geometry. More precisely, we start with the projective space and choose as the group of rigid motions those projectivies that leave the so-called absolute figure invariant. The space $\mathbb{I}^3$ is the geometry in affine space corresponding to the choice of an absolute figure composed of a plane and a degenerated quadric \cite{Sachs1990}. (See, e.g., Ref. \cite{daSilvaTJM2020,PottmannACM2009} for texts in English.) Here, we shall adopt the metric viewpoint. In other words, let us denote by the \emph{simply isotropic space}, $\mathbb{I}^3$, the vector space $\mathbb{R}^3$ equipped with the degenerated metric
\begin{equation}
    \langle u,v\rangle = u^1v^1+u^2v^2,
\end{equation}
where $u=(u^1,u^2,u^3)$ and $v=(v^1,v^2,v^3)$. The set of vectors $\{(0,0,u^3)\}$ that degenerate the isotropic metric is the set of \emph{isotropic vectors}. A plane containing an isotropic vector is called an \emph{isotropic plane}. On the set of isotropic vectors we use the secondary metric $\llangle u,v\rrangle=u^3v^3$. (Therefore, $\mathbb{I}^3$ is an example of a Cayley-Klein vector space \cite{StruveRM2005}.) The inner product induces an isotropic semi-norm in a natural way: $\Vert u\Vert=\sqrt{\langle u,u\rangle}$. In addition, we shall refer to the projection over the $xy$-plane as the \emph{top-view projection}, which is denoted here by $u=(u^1,u^2,u^3)\mapsto \tilde{u}\equiv(u^1,u^2,0)$.

In the following, it will prove useful to use the Euclidean scalar product. Let us denote by $\cdot$ and $\times$ the inner and vector products in Euclidean space $\mathbb{E}^3$: $u\cdot v=u^1v^1+u^2v^2+u^3v^3$ and $u\times v=(u^2v^3-u^3v^2,-u^1v^3+u^3v^1,u^1v^2-u^2v^1)$.

We are interested on surfaces $\mathbf{x}:(u^1,u^2)\in U\mapsto M\subset\mathbb{I}^3$ whose induced metric is non-degenerated. The \emph{admissible surfaces} are those surfaces $M^2$ that do not have any isotropic tangent plane. Therefore, $\vert\partial(x^1,x^2)/\partial(u^1,u^2)\vert\not=0$ and, consequently, every admissible surface can be reparameterized as a graph $(u^1,u^2)\mapsto (u^1,u^2,F(u^1,u^2))$, the so-called  \emph{normal form}. Here, the induced first fundamental form reads $\mathrm{I}=(\rmd u^1)^2+(\rmd u^2)^2$, which implies that every surface in isotropic space is intrinsically flat. However, it is possible to introduce an extrinsic Gaussian curvature as the ratio between the determinants of the first and second fundamental forms. Indeed, the normal of a surface with respect to the metric is the vector field $\mathcal{N}=(0,0,1)$. Then, we may define the Christoffel symbols $\Gamma_{ij}^k$ and the coefficients of the second fundamental form $h_{ij}$ by the equation
\[
\mathbf{x}_{ij} = \Gamma_{ij}^k\mathbf{x}_k+h_{ij}\mathcal{N},
\]
where $\mathbf{x}_{i}=\partial \mathbf{x}/\partial u^i$, $\mathbf{x}_{ij}=\partial^2\mathbf{x}/\partial u^i\partial u^j$, and we are summing on repeated indices (from 1 to 2). Finally, the isotropic Gaussian and mean curvatures are respectively defined by
\begin{equation}
    K = \frac{h}{g}=\frac{h_{11}h_{22}-h_{12}^2}{g_{11}g_{22}-g_{12}^2}\mbox{ and }H = \frac{g_{11}h_{22}-2g_{12}h_{12}+g_{22}h_{11}}{2(g_{11}g_{22}-g_{12}^2)},
\end{equation}
where $g_{ij}=\langle \mathbf{x}_i,\mathbf{x}_j\rangle$ denotes the coefficients of the first fundamental form. 

The coefficients of the second fundamental form can be computed as
\begin{equation}
    h_{ij}=\frac{\det(\mathbf{x}_1,\mathbf{x}_2,\mathbf{x}_{ij})}{\sqrt{g_{11}g_{22}-g_{12}^2}}=\mathbf{x}_{ij}\cdot\mathbf{N}_m,\,\mbox{ where }\mathbf{N}_m = \frac{\mathbf{x}_1\times \mathbf{x}_2}{\sqrt{g_{11}g_{22}-g_{12}^2}}.
\end{equation}
We shall refer to $\mathbf{N}_m$ as the \emph{minimal normal} since the corresponding shape operator $-\rmd\mathbf{N}_m$ is traceless \cite{KelleciJMAA2021}. Its determinant, however, gives the Gaussian curvature, $K=\det(-\rmd\mathbf{N}_m)$. 

\begin{remark}
{Rigorously, the minimal normal $\mathbf{N}_m$ does not provide a proper Gauss map from which one can define a shape operator understood as a self-adjoint operator on the tangent planes. Therefore, the trace and determinant of $-\rmd\mathbf{N}_m$ are computed with the proviso that the tangent vector $\mathbf{x}_i$ is formally identified with $\mathbf{a}_i=\mathbf{x}_i\times\mathcal{N}$. (Please, see Subsection 2.1 and Remark 2.1 of \cite{KelleciJMAA2021}.)}
\end{remark}

If we insist on seeing $K$ and $H$ as the determinant and trace of a shape operator, we may introduce the so-called \emph{parabolic normal} $\xi$ defined by \cite{daSilvaJG2019}
\begin{equation}
    \xi = \tilde{\mathbf{N}}_m+\frac{1}{2}\left(1-\Vert \tilde{\mathbf{N}}_m\Vert^2\right)\mathcal{N}.
\end{equation}
The parabolic normal $\xi$ takes values on the unit sphere of parabolic type $\Sigma^2=\{(x,y,z)\in\mathbb{I}^3:z=\frac{1}{2}(1-x^2-y^2)\}$. From the \emph{isotropic shape operator} $A=-\rmd\xi$, we can compute the Gaussian and mean curvatures as $K=\det A$ and $H=\mathrm{tr}\,A$ \cite{daSilvaJG2019,PavkovicJAZU1990}. In addition, we can alternatively compute the coefficients of the second fundamental form as $h_{ij}=\langle A(\mathbf{x}_i),\mathbf{x}_j\rangle$. 

\begin{remark}
{In addition to spheres of parabolic type, we also have the so-called {\emph{spheres of cylindrical  type}}, which are the metric spheres: $G(p,r)=\{x\in\mathbb{I}^3:\langle x-p,x-p\rangle=r^2\}$. These surfaces, however, are not admissible since all their tangent planes are isotropic. Therefore, we can not use these spheres to define a Gauss map for admissible surfaces as done for the parabolic normal.} 
\end{remark}

\subsection{Simply isotropic minimal surfaces}

{We define as minimal surfaces in simply space those surfaces with $H=0$.} If $M^2$ is parameterized in its normal form, $\mathbf{x}(u^1,u^2)=(u^1,u^2,F(u^1,u^2))$, then the minimal and parabolic normals are given by $\mathbf{N}_m=(-F_1,-F_2,1)$ and $\xi=(-F_1,-F_2,\frac{1}{2}-\frac{1}{2}(F_1^2+F_2^2))$. The first and second fundamental forms are $\mathrm{I}=\delta_{ij}\rmd u^i\rmd u^j$ and $\mathrm{II}=F_{ij}\rmd u^i\rmd u^j$, from which follows that the Gaussian and mean curvatures are $K=F_{11}F_{22}-F_{12}^2$ and $H=\frac{1}{2}(F_{11}+F_{22})$. Therefore, 
\begin{theorem}
Let $M^2\subset\mathbb{I}^3$ be an admissible simply isotropic minimal surface, then $M^2$ is locally the graph of a harmonic function. 
\end{theorem}

{In Euclidean space, the minimal surfaces are the critical points of the area functional, i.e., $H=0$ is the Euler-Lagrange equation of the problem $\min_M\int_M\rmd A$. The same can not be done in the simply isotropic space since the area in the simply isotropic induced metric is the same as its projection on top view plane. In fact, if we fix the boundary curve $\gamma$, every surface $M^2$ with $\partial M^2=\gamma$ would be a critical point of the simply isotropic area functional. In the simply isotropic space, instead of the area computed from the isotropic metric, we may consider the so-called \emph{relative area} $\mathcal{O}^*$ \cite{Sachs1990,StrubeckerMZ1942}:}
\begin{equation}
    \mathcal{O}^* = \int_{U}\det(\xi,\mathbf{x}_1,\mathbf{x}_2)\,\rmd u^1\rmd u^2 = \int_{U} \xi\cdot\mathbf{N}_m \sqrt{\det g_{ij}}\,\rmd u^1\rmd u^2 = \int_{U} \xi\cdot\mathbf{N}_m\rmd A.
\end{equation}
{If $M^2$ is parametrized in its normal form over $U$, then the relative area becomes $\mathcal{O}^*=\frac{1}{2}\int_U(1+F_1^2+F_2^2)\rmd u^1\rmd u^2$. Now, consider a normal variation of $M^2$: $\mathbf{x}_{\varepsilon}=\mathbf{x}+\varepsilon  V\mathcal{N}=(u^1,u^2,F+\varepsilon V)$, where $V\vert_{\partial M^2}=0$. The relative area as a function of $\varepsilon$ is then}
\begin{eqnarray}
\mathcal{O}^*(\varepsilon) 
& = & \frac{1}{2}\int_U [1+F_1^2+F_2^2+2\varepsilon(V_1F_1+V_2F_2)+\varepsilon^2(V_1^2+V_2^2)]\,\rmd u^1\rmd u^2\nonumber\\
& = & \mathcal{O}^*(0)-\varepsilon\int_U V(F_{11}+F_{22})\,\rmd u^1\rmd u^2+\frac{\varepsilon^2}{2}\int_U(V_1^2+V_2^2)\,\rmd u^1\rmd u^2\nonumber\\
& = & \mathcal{O}^*(0)-2\varepsilon\int_U VH\,\rmd A+\frac{\varepsilon^2}{2}\int_U(V_1^2+V_2^2)\,\rmd u^1\rmd u^2.
\end{eqnarray}
{Therefore, a surface $M^2$ is a critical point of the relative area functional if and only if $H=0$. Note, however, that despite the Euler-Lagrange equation $H=0$ is a simply isotropic invariant, the relative area itself is not. (Note, in addition, that every simply isotropic minimal surface is stable, i.e., they all have positive second variation.)}

{An example of simply isotropic minimal surface is given by the helicoid, which is the graph of the harmonic function $F(x,y)=a \arctan\frac{y-y_0}{x-x_0}$. Since the helicoid is the only Euclidean minimal surface which is simultaneously the graph of a harmonic function, the helicoid is the only surface which is both minimal in $\mathbb{E}^3$ and $\mathbb{I}^3$. Further} examples of isotropic minimal surfaces can be provided by (i) looking for isotropic analogs of Euclidean minimal surfaces, such as the well known Enneper and Scherk surfaces \cite{StrubeckerAMSUH1975,StrubeckerCrelle1954}; (ii) employing some sort of separation of variables, such as the Scherk surfaces, which correspond to solutions of the form $z=f(x)+g(y)$ or $x=g(y)+h(z)$, or the so-called affine factorable surfaces \cite{AydinTWMS2020}, which correspond to solutions of the form $z=f(x)g(y+ax)$ or $x=g(y+az)h(z)$ ($a$ constant); or (iii) looking for surfaces invariant by a 1-parameter group of isotropic isometries \cite{daSilvaMJOU2021,Sachs1990}. Here, we shall be interested on a generic representation of isotropic minimal surfaces in terms of holomorphic functions.

\section{Stereographic projection in simply isotropic space}

In order to associate the holomorphic data of a Weierstrass representation with a choice of a Gauss map, we first define the stereographic projection of either a unit sphere of parabolic type or a horizontal plane over the (top view) plane. The first will be related to the parabolic normal, since it takes values on a unit sphere of parabolic type, while the second projection will be related to the minimal normal, since it takes values on the horizontal plane $\{z=1\}$.

Let $\Sigma^2$ be the isotropic unit sphere of parabolic type centered at the {origin\footnote{{By the center of the sphere $\Sigma^2$ we mean its focus.}}} 
\begin{equation}
    \Sigma^2 = \left\{(x,y,z)\in\mathbb{I}^3:z=\frac{1}{2}-\frac{1}{2}(x^2+y^2)\right\}.
\end{equation}
The North pole of $\Sigma^2$ is the point $N=(0,0,\frac{1}{2})$. (We may say that the South pole is the point of the sphere at infinity.) As in Euclidean space, we define the stereographic projection $\pi:\Sigma^2\setminus\{N\}\to\mathbb{C}^*$ by defining $\pi(p)$ as the intersection of the line connecting $N$ to $p$ with the $xy$-plane, which is identified with $\mathbb{C}$. (Here, $\mathbb{C}^*=\mathbb{C}\setminus\{0\}$.)  

The line $\ell$ connecting $N$ to $p=(p_1,p_2,p_3)\in\Sigma^2$ can be parameterized as $\ell:t\mapsto t(p-N)+N=(tp_1,tp_2,t(p_3-\frac{1}{2})+\frac{1}{2})$. Imposing the third coordinate of $\ell(t)$ to vanish, gives
\begin{equation}
    0=t(p_3-\frac{1}{2})+\frac{1}{2}\Rightarrow t=\frac{1}{1-2p_3}.
\end{equation}
Then, the stereographic projection $\pi$ is given by
\begin{equation}
    p\in\Sigma^2\setminus\{N\} \mapsto \pi(p)=\left(\frac{p_1}{1-2p_3},\frac{p_2}{1-2p_3}\right)=\left(\frac{p_1}{p_1^2+p_2^2},\frac{p_2}{p_1^2+p_2^2}\right).
\end{equation}

On the other hand, given $\pi(p)=(x,y,0)\in\mathbb{C}^*$, the line {$\lambda$} connecting $N$ to $\pi(p)$ is ${\lambda}:t\mapsto (\pi(p)-N)t+N=(xt,yt,\frac{1}{2}(1-t))$. Imposing ${\lambda}(t)\in\Sigma^2$, gives
\begin{equation}
    \frac{1}{2}-\frac{t}{2}=\frac{1}{2}-\frac{1}{2}[(xt)^2+(yt)^2]\Rightarrow t = t^2(x^2+y^2).
\end{equation}
Since $t\not=0$, we find $t=(x^2+y^2)^{-1}$. Thus, the inverse $\pi^{-1}$ of the stereographic projection is given by
\begin{equation}
    (x+\rmi y)\in\mathbb{C}^* \mapsto \pi^{-1}(x+\rmi y)=\left(\frac{x}{x^2+y^2},\frac{y}{x^2+y^2},\frac{1}{2}-\frac{1}{2(x^2+y^2)}\right).
\end{equation}
Note that the origin $0=0+\rmi0$ would be sent by $\pi^{-1}$ to a point on $\Sigma^2$ at infinity, which we may see as the South pole $S$ of $\Sigma^2$: $S\equiv\pi^{-1}(0)$. In short, we can define

\begin{definition}[Parabolic stereographic projection]
Let $\Sigma^2$ be the unit sphere of parabolic type in $\mathbb{I}^3$ centered at the origin and $N=(0,0,\frac{1}{2})$ its North pole. The stereographic projection, $\pi$, of $\Sigma^2$ over the {$xy$-plane}, identified with $\mathbb{C}$, is the map
$$
\begin{array}{ccccl}
    \pi & : & \Sigma^2\setminus\{N\} & \to     & \mathbb{C} \\
        &    & (p_1,p_2,p_3)  & \mapsto & \Big(\displaystyle\frac{p_1}{1-2p_3},\displaystyle\frac{p_2}{1-2p_3}\Big)=\displaystyle\frac{1}{p_1^2+p_2^2}(p_1,p_2). \\
\end{array}
$$
In addition, its inverse is the map given by
$$
\begin{array}{ccccl}
    \pi^{-1} & : & \mathbb{C}\setminus\{0\} & \to     & \Sigma^2 \\
        &    & x+\rmi y  & \mapsto & \Big(\displaystyle\frac{x}{x^2+y^2},\frac{y}{x^2+y^2},\frac{1}{2}-\displaystyle\frac{1}{2(x^2+y^2)}\Big). \\
\end{array}
$$
\end{definition}

Before studying the properties of the parabolic stereographic projection, let us define the stereographic projection of the plane $\Pi:z=1$ over the $xy$-plane. First, note that for a sphere of parabolic type of radius $R$, $\Sigma^2_R:z=\frac{R}{2}-\frac{1}{2R}(x^2+y^2)$, the corresponding stereographic projection is given by 
\begin{equation}
    \pi_R(p_1,p_2,p_3) = \Big(\frac{p_1}{1-\frac{2}{R}p_3},\frac{p_2}{1-\frac{2}{R}p_3}\Big)=\frac{R^2}{p_1^2+p_2^2}(p_1,p_2).
\end{equation}
In addition, for $R\gg1$, we have $\Sigma_R^2\sim\{z=\frac{R}{2}\}$ and $\pi_R(p)\sim \tilde{p}$. In other words, the sphere $\Sigma_R^2$ can be approximated by a plane parallel to the top view plane. This reasoning suggests defining the stereographic projection associated with the plane $\{z=1\}$ by the top view projection $\pi_{\infty}(p)=\tilde{p}$.

\begin{definition}[Top view projection as a stereographic projection]
Let $\Pi$ be the plane $\{(x,y,z):z=1\}$. The stereographic projection, $\pi_{\infty}$, of $\Pi$ over the {$xy$-plane}, identified with $\mathbb{C}$, is the map
$\pi_{\infty}(p_1,p_2,p_3)=(p_1,p_2)$. In addition, its inverse is simply given by $\pi_{\infty}^{-1}(x+\rmi y)= (x,y,1)$.
\end{definition}

For the stereographic projection of the Euclidean sphere over the plane, $\pi_E$, it is well known that great and small circles are sent under $\pi_E$ to circles or lines in $\mathbb{C}$ and vice versa \cite{Ahlfors1979}. We have an analogous result in isotropic space. To prove this, we first investigate the image of spherical $r$-geodesics under the parabolic stereographic projection in order to single out the properties characterizing their image in $\mathbb{C}$ and, later,  we investigate the image of generic plane curves on the parabolic sphere. (A curve on a surface $M^2$ is an $r$-geodesic if its acceleration vector is parallel to the parabolic normal of $M^2$ \cite{daSilvaJG2019,PavkovicJAZU1990}. It is worth mentioning that $r$-geodesics on a sphere of parabolic type come from the intersections with planes passing through its center \cite{daSilvaJG2019}.)  

\begin{proposition}\label{PropStereogProjRelGeod}
A curve $\alpha$ is an $r$-geodesic in $\Sigma^2$ if and only if, under the stereographic projection $\pi$,  it corresponds either to a straight line in $\mathbb{C}$ passing through the origin if $N\in\alpha$ or to a circle in $\mathbb{C}$ whose radius $R$ and center $\mathcal{C}$ satisfy $R^2=1+\mathrm{dist}(\mathcal{C},0)^2$ if $N\not\in\alpha$, where $N$ is the north pole of $\Sigma^2$.
\end{proposition}
\begin{proof}
Since any $r$-geodesic $\alpha$ in $\Sigma^2$ comes from the intersection with a plane $\Pi$ passing through the center of $\Sigma^2$ \cite{daSilvaJG2019}, we can implicitly write $\alpha$ as
\begin{equation}\label{eqImpEqParR-geod}
    \alpha:\left\{
    \begin{array}{c}
         z = \frac{1}{2}-\frac{1}{2}(x^2+y^2)  \\
         Ax+By+Cz=0
    \end{array}
    \right.,
\end{equation}
where $(A,B,C)\not=0$. Note that $(x,y,z)\in\alpha$ is sent to $\pi(\alpha)=\frac{1}{x^2+y^2}(x,y)$. 

Substituting the first expression of Eq. (\ref{eqImpEqParR-geod}) into the second, one has
$$2Ax+2By+C(1-x^2-y^2)=0.$$
If $C=0$, then $2Ax+2By=0\Rightarrow \frac{x}{x^2+y^2}A+\frac{y}{x^2+y^2}B=0$ and $\pi(\alpha)$ lies in a straight line in $\mathbb{C}$ passing through $0\in\mathbb{C}$. (Note that $0\not\in\pi(\alpha)$.) On the other hand, if $C\not=0$, then  $2Ax/(x^2+y^2)+2By/(x^2+y^2)+C/(x^2+y^2)=C$. Here, $N\not\in\alpha$ and, in addition, writing $(X,Y)=\pi(\alpha)=\frac{1}{(x^2+y^2)}(x,y)$, it follows that $X^2+Y^2=(x^2+y^2)^{-1}$. Finally,  
$$C(X^2+Y^2)+2AX+2BY=C\Rightarrow(X+\frac{A}{C})^2+(Y+\frac{B}{C})^2=1+\frac{A^2+B^2}{C^2},$$
which is the circle of center $\mathcal{C}=(A/C,B/C)$ and radius $R=\sqrt{1+\mathrm{dist}(\mathcal{C},0)^2}$.

Conversely, given $\ell:AX+BY=0$, we may parameterize $\ell^*\equiv\ell-\{0\}$ by $t\mapsto (t,-At/B)$, where we are assuming, without loss of generality, that $B\not=0$. Finally, $\ell^*$ is sent by $\pi^{-1}$ into $$t\mapsto \pi^{-1}(\ell^*(t))=\Big(\frac{B^2}{t(A^2+B^2)},-\frac{AB}{t(A^2+B^2)},\frac{1}{2}-\frac{B^2}{2t^2(A^2+B^2)}\Big),$$
which lies in the plane $\Pi:Ax+By+Cz=0$ with $C=0$. On the other hand, given a circle $c:(X-a)^2+(Y-b)^2=R^2=1+a^2+b^2$ in  $\mathbb{C}$, we may parameterize it by $\theta\mapsto (a+R\cos\theta,b+R\sin\theta).$ Applying $\pi^{-1}$ gives
$$\theta\mapsto\pi^{-1}(c(\theta))=\Big(\frac{a+R\cos\theta}{X^2+Y^2},\frac{b+R\sin\theta}{X^2+Y^2},\frac{1}{2}-\frac{1}{2(X^2+Y^2)}\Big),$$
where $X^2+Y^2=1+2(a^2+b^2)+2aR\cos\theta+2bR\sin\theta$. Direct computations show that $\pi^{-1}(c(\theta))$ lies in the plane $\Pi:-ax-by+z=0$.
\qed
\end{proof}

\begin{proposition}
A curve $\alpha$ in $\Sigma^2$ is a small circle, i.e., a curve obtained from the intersection of $\Sigma^2$ with a plane, if and only if it is sent by the parabolic stereographic project $\pi$ into either a circle in $\mathbb{C}$  if $N\not\in\alpha$ or into a straight line in $\mathbb{C}$ if $N\in\alpha$, where $N$ is the north pole of $\Sigma^2$.
\end{proposition}
\begin{proof}
Since any plane curve $\alpha$ in $\Sigma^2$ comes from the intersection with a plane $\Pi$ with normal $(A,B,C)\not=0$, we can implicitly write $\alpha$ as
\begin{equation}\label{eqParSphericalPlaneCurves}
    \alpha:\left\{
    \begin{array}{c}
         z = \frac{1}{2}-\frac{1}{2}(x^2+y^2)  \\
         Ax+By+Cz+D=0
    \end{array}
    \right..
\end{equation}
Note that if $(x,y,z)\in\alpha$, then $\pi(\alpha)=(\frac{x}{x^2+y^2},\frac{y}{x^2+y^2})$.

Now, substituting the first expression of Eq. (\ref{eqParSphericalPlaneCurves}) into the second, one has
\begin{equation}\label{EqIntParSphrWithPlaneSingleEq}
    2Ax+2By+C(1-x^2-y^2)+2D=0.
\end{equation}
If $C=0$, then $2Ax+2By+2D=0$. Since $C=0$, $N\in\alpha$ if and only if $D=0$. Then, $Ax+By=0$ and $\pi(\alpha)$ lies in a straight line passing through $0$. Otherwise, if $N\not\in\alpha$, then $Ax+By+D=0\Rightarrow A\frac{x}{x^2+y^2}+B\frac{y}{x^2+y^2}+\frac{D}{x^2+y^2}=0$. Writing $(X,Y)=\pi(\alpha)=\frac{1}{(x^2+y^2)}(x,y)$, it follows that $X^2+Y^2=(x^2+y^2)^{-1}$ and $\pi(\alpha)$ lies in the circle
$$\Big(X+\frac{A}{2D}\Big)^2+\Big(Y+\frac{B}{2D}\Big)^2=\frac{A^2+B^2}{4D^2}>0.$$

On the other hand, if $C\not=0$, then  $(C+2D)+2Ax+2By=C(x^2+y^2)$. Here, $N\in\alpha$ if and only if $C+2D=0$. So, if $N\in\alpha$, then $A\frac{x}{x^2+y^2}+B\frac{y}{x^2+y^2}=C$ and $\pi(\alpha)$ lies in a straight line not passing through the origin. Otherwise, if $N\not\in\alpha$, then writing $(X,Y)=\pi(\alpha)=\frac{1}{(x^2+y^2)}(x,y)$, it follows that Eq. (\ref{EqIntParSphrWithPlaneSingleEq}) leads to $(C+2D)(X^2+Y^2)+2AX+2BY=C$. Finally, $\pi(\alpha)$ lies in the circle
$$\Big(X+\frac{A}{C+2D}\Big)^2+\Big(Y+\frac{B}{C+2D}\Big)^2=\frac{C}{C+2D}+\frac{A^2+B^2}{(C+2D)^2}>0,$$
where the right-hand side of the equation above has to be positive because  $\Sigma^2\cap\Pi\not=\emptyset$, or just a single point. Indeed, since $C\not=0$, $\Pi$ can not be vertical and, consequently, there exists a point $p_0\in \Sigma^2$ at which $T_{p_0}\Sigma^2$ is parallel to $\Pi$. Since the implicit equation of $T_{p}\Sigma^2$ at any $p=(p_1,p_2,p_3)$ is $z-p_3=-p_1(x-p_1)-p_2(y-p_2)$, parallelism occurs at $p_0=(\frac{A}{C},\frac{B}{C},\frac{1}{2}-\frac{1}{2C^2}(A^2+B^2))$ since the equation of  $T_{p_0}\Sigma^2$ is $Ax+By+Cz=\frac{1}{2C}(A^2+B^2+C^2)$. Finally, assuming without loss of generality that $C>0$, for the intersection $\alpha=\Sigma^2\cap\Pi$ to be non-empty, we should have $-D<\frac{1}{2C}(A^2+B^2+C^2)$, i.e., $A^2+B^2+C(C+2D)>0$. (Equality would occur when $\Pi=T_{p_0}\Sigma^2$ and, consequently, $\alpha$ would degenerate to a single point $\alpha=\{p_0\}$.) Indeed, if two parallel planes $\Pi_i:ax+by+cz=d_i$ ($i=1,2$) have a normal $\mathbf{n}=(a,b,c)$ making an acute angle with $\mathcal{N}=(0,0,1)$, i.e., $c>0$, then $d_2>d_1$ implies that $\Pi_2$ is above $\Pi_1$ with respect to $\mathcal{N}$.

For the converse, the reader may follow similar steps to those of the proof of Prop. \ref{PropStereogProjRelGeod} to show that circles/lines in $\mathbb{C}$ are sent under $\pi^{-1}$ to small circles in $\Sigma^2$.
\qed
\end{proof}

\section{Weierstrass representation of simply isotropic minimal surfaces}

In analogy with what happens in $\mathbb{E}^3$,  if we apply the Laplace-Beltrami operator $\Delta_g$ to the parameterization $\mathbf{x}$ of a minimal surface $M^2$ in $\mathbb{I}^3$, we have $\Delta_g\mathbf{x}=2H\mathcal{N}$ \cite{SatoArXiv2018}. Therefore, if we parameterize $M^2$ with isothermal coordinates, i.e., $g_{11}=g_{22}$ and $g_{12}=0$, then the coordinates functions of $M^2$ are harmonic functions on the plane. (Remember, if $g_{ij}=F^2\delta_{ij}$, then $\Delta_g=\frac{1}{F^2}
\Delta$, where $\Delta$ is the flat 2d Laplacian.) Identifying the Euclidean plane with $\mathbb{C}$, we can then parameterize $M^2$ using the real part of holomorphic functions. More precisely, we can parameterize $M^2$ by  \cite{SatoArXiv2018}
\begin{equation}
    \mathbf{x}(z) = \re\left(\int_z\phi(w)\rmd w\right),
\end{equation}
where $\phi=(\phi_1,\phi_2,\phi_3)\in\mathbb{C}^3$ satisfies $\phi_1^2+\phi_2^2=0$ and $\vert\phi_1\vert^2+\vert\phi_2\vert^2=0$. The first condition guarantees $\mathbf{x}$ above is an isothermic simply isotropic minimal immersion, while the second guarantees $\mathbf{x}$ is admissible. Note that, due to the degenerate nature of the simply isotropic metric, the third coordinate $\phi_3$ is not functionally related to $\phi_1$ and $\phi_2$. In the following, we shall exploit this freedom to associated a Weierstrass representation with a given choice of a Gauss map.

If we write $\mathbf{x}=\re(\int\mathbf{\phi})=\frac{1}{2}(\int\mathbf{\phi}+\int\bar{\mathbf{\phi}})$, $\phi=(\phi_1,\phi_2,\phi_3)$, then using that $\partial_u=\partial_z+\partial_{\bar{z}}$ and $\partial_v=\rmi(\partial_z-\partial_{\bar{z}})$ \cite{Ahlfors1979}, $z=u+\rmi v$, we write the tangent vectors as
\begin{equation}
    \mathbf{x}_u = \displaystyle\frac{\mathbf{\phi}+\bar{\mathbf{\phi}}}{2}=\re(\mathbf{\phi})\mbox{ and }
    \mathbf{x}_v = -\displaystyle\frac{\mathbf{\phi}-\bar{\mathbf{\phi}}}{2\rmi}=-\im(\mathbf{\phi}).
\end{equation} 
Now, noticing that 
\begin{equation}
\mathbf{x}_u\times\mathbf{x}_v = (\mathrm{Im}(\phi_2\bar{\phi}_3),\mathrm{Im}(\phi_3\bar{\phi}_1),\mathrm{Im}(\phi_1\bar{\phi}_2)),
\end{equation}
the minimal normal is 
\begin{equation}
    \mathbf{N}_m = \left(\frac{\mathrm{Im}(\phi_2\bar{\phi}_3)}{\mathrm{Im}(\phi_1\bar{\phi}_2)},\frac{\mathrm{Im}(\phi_3\bar{\phi}_1)}{\mathrm{Im}(\phi_1\bar{\phi}_2)},1\right).
\end{equation}
In addition, since $0=\phi_1^2+\phi_2^2=(\phi_2-\rmi\phi_1)(\phi_2+\rmi\phi_1)$, we can write $\phi_2=\pm\rmi\phi_1$. The minimal normal then becomes
\begin{equation}
    \mathbf{N}_m = \left(\frac{\mathrm{Im}(\pm\rmi\phi_1\bar{\phi}_3)}{\mathrm{Im}(\mp\rmi\vert\phi_1\vert^2)},\frac{\mathrm{Im}(\phi_3\bar{\phi}_1)}{\mathrm{Im}(\mp\rmi\vert\phi_1\vert^2)},1\right) = \left(-\frac{\mathrm{Re}(\phi_3\bar{\phi}_1)}{\vert\phi_1\vert^2},\mp\frac{\mathrm{Im}(\phi_3\bar{\phi}_1)}{\vert\phi_1\vert^2},1\right).
\end{equation}
Finally, seeing the minimal normal $\mathbf{N}_m$ and the parabolic normal $\xi$ as maps taking values in $\mathbb{C}\times\mathbb{R}$, we can either write
\begin{equation}
    \mathbf{N}_m = \left(-\frac{\phi_3}{\phi_1},1\right)\mbox{ and }\xi = \left(-\frac{\phi_3}{\phi_1},\frac{1}{2}-\frac{1}{2}\left\vert\frac{\phi_3}{\phi_1}\right\vert^2\right),\,\mbox{ if }\phi_2=\rmi\phi_1,
\end{equation}
or 
\begin{equation}
    \mathbf{N}_m = \left(-\frac{\bar{\phi}_3}{\bar{\phi}_1},1\right)\mbox{ and }\xi = \left(-\frac{\bar{\phi}_3}{\bar{\phi}_1},\frac{1}{2}-\frac{1}{2}\left\vert\frac{\phi_3}{\phi_1}\right\vert^2\right),\,\mbox{ if }\phi_2=-\rmi\phi_1.
\end{equation}

Note that there is a freedom in the choice of the third coordinate of the complex curve $\phi$. Consequently, this allows us to choose the second holomorphic function $G$ to conform to the choice of Gauss map we have in mind. In the following, it will prove to be more convenient to choose $\phi_2=\rmi\phi_1$ when working with the minimal normal $\mathbf{N}_m$ and to choose $\phi_2=-\rmi\phi_1$ when working with the parabolic normal $\xi$. We should do that in order to have two holomorphic functions in the Weierstrass data instead of one holomorphic and another anti-holomorphic.

On the one hand, choosing $\phi_2=\rmi \phi_1$ and by requiring that the top view projection of the minimal normal $\mathbf{N}_m$ is part of the Weierstrass data, we can write for some holomorphic functions $F$ and $G$
\begin{equation}\label{eq::WeierstrassRepWithGtopviewNm}
    \phi = (F,\rmi F,-FG) \Rightarrow \tilde{\mathbf{N}}_m = G.
\end{equation}
From $\vert\phi_1\vert^2+\vert\phi_2\vert^2=2\vert F\vert^2$, it follows that $\mathbf{x}$ fails to be regular on the zeros of $F$. This also means that the singularities of a simply isotropic minimal immersion are isolated.

On the other hand, choosing $\phi_2=-\rmi \phi_1$ and by requiring that the  stereographic projection of the parabolic normal $\xi$ is part of the Weierstrass data, we can write for some holomorphic functions $F$ and  $G$
\begin{equation}\label{eq::WeierstrassRepWithGprojParNor}
    \phi = \left(F,-\rmi F,-\displaystyle\frac{F}{G}\right), 
\end{equation}
which gives 
\begin{equation}
    \xi = \left(\frac{G}{\vert G\vert^2},\frac{1}{2}-\frac{1}{2\vert G\vert^2}\right)\Rightarrow \pi(\xi) = G.
\end{equation}
As before, we also have $\vert\phi_1\vert^2+\vert\phi_2\vert^2=2\vert F\vert^2$, from which follows that $\mathbf{x}$ fails to be regular on the zeros of $F$. 

Conversely, on the one hand, suppose we have a meromorphic function $G$ and a holomorphic function $F$ defined on $M^2$ such that the zeros of $F$ coincide with the poles of $G$ in a way that a zero of order $m$ of $F$ corresponds to a pole of order $m$ of $G$. Then, $\phi_1=F$, $\phi_2=\rmi F$, and $\phi_3=-FG$ are holomorphic on $M^2$ and satisfy $\phi_1^2+\phi_2^2=0$ and $\vert\phi_1\vert^2+\vert\phi_2\vert^2>0$ outside the zeros of $F$. In addition, if $\phi_1,\phi_2,$ and $\phi_3$ have no real periods we obtain a simply isotropic minimal immersion $\mathbf{x}:M^2\to\mathbb{I}^3$ such that $\tilde{\mathbf{N}}_m=G$. On the other hand, suppose we have holomorphic functions $G$ and $F$ defined on $M^2$ such that the zeros of $F$ coincide with the zeros of $G$ in a way that a zero of order $m$ of $F$ corresponds to a zero of order $m$ of $G$. Then, $\phi_1=F$, $\phi_2=-\rmi F$, and $\phi_3=-F/G$ are holomorphic on $M^2$ and satisfy $\phi_1^2+\phi_2^2=0$ and $\vert\phi_1\vert^2+\vert\phi_2\vert^2>0$ outside the zeros of $F$. In addition, if $\phi_1,\phi_2,$ and $\phi_3$ have no real periods we obtain a simply isotropic minimal immersion $\mathbf{x}:M^2\to\mathbb{I}^3$ such that $\pi(\xi)=G$.

\subsection{Extrinsic geometry of simply isotropic minimal surfaces}

Now, we shall focus on extrinsic geometric properties. For that, we must compute the second fundamental form of a minimal immersion $\mathbf{x}=\re\int\phi$. Since the tangent vectors are
\begin{equation}\label{eq::TangVecGenericHolRep}
    \mathbf{x}_u = \frac{\phi+\bar{\phi}}{2}=\re(\phi)\mbox{ and }\mathbf{x}_v = -\frac{\phi-\bar{\phi}}{2\rmi}=-\im(\phi),
\end{equation}
it follows that the second derivatives are
\begin{equation}\label{eq::2ndDerGenericHolRep}
    \mathbf{x}_{uu} = \displaystyle\frac{\mathbf{\phi}'+\bar{\mathbf{\phi}}'}{2}=\re(\mathbf{\phi}'),\,\mathbf{x}_{uv} = -\displaystyle\frac{\mathbf{\phi}'-\bar{\mathbf{\phi}}'}{2\rmi}=-\im(\mathbf{\phi}'),\,
    \mathbf{x}_{vv} = -\re(\mathbf{\phi}').
\end{equation}
Finally, we are in condition to compute the second fundamental form and, from it, the Gaussian curvature of a simply isotropic minimal immersion. But, first, we need the following auxiliary result.

\begin{lemma}\label{lem::1st2ndFFandKminComplexCurve}
The first and second fundamental forms $\mathrm{I}$ and $\mathrm{II}$ and the Gauss curvature $K$ of the simply isotropic minimal immersion associated with the complex curve $\phi=(\phi_1,\phi_2=\pm\rmi\phi_1,\phi_3)$ are given by
\begin{equation}
    \mathrm{I} = \vert \phi_1\vert^2(\rmd u^2+\rmd v^2),
\end{equation}
\begin{equation}
    \mathrm{II} = \re\left[\phi_1\left(\frac{\phi_3}{\phi_1}\right)'\right](\rmd u^2-\rmd v^2)-2\,\im\left[\phi_1\left(\frac{\phi_3}{\phi_1}\right)'\right]\rmd u\rmd v,
\end{equation}
and
\begin{equation}
    K = - \left\vert\frac{1}{ \phi_1}\left(\frac{\phi_3}{\phi_1}\right)'\right\vert^2,
\end{equation}
respectively.
\end{lemma}

From Lemma \ref{lem::1st2ndFFandKminComplexCurve}, the proof of the next theorem follows straightforwardly.

\begin{theorem}\label{thr::1st2ndFFandKMinNormal}
The first and second fundamental forms and the Gauss curvature of the minimal immersion associated with $\phi=(F,\rmi F,-FG)$, so that $G=\tilde{\mathbf{N}}_m$, are given by
\begin{equation}
    \mathrm{I} = \vert F\vert^2\,\vert\rmd z\vert^2,\,    \mathrm{II} =-\re(FG'\,\rmd z^2), \mbox{ and }    K = - \left\vert\frac{G'}{F}\right\vert^2,
\end{equation}
respectively. On the other hand, the first and second fundamental form and Gauss curvature of the minimal immersion associated with $\phi=(F,-\rmi F,-\frac{F}{G})$, so that $G=\pi(\xi)$, are given by
\begin{equation}
    \mathrm{I} = \vert F\vert^2\,\vert\rmd z\vert^2,\,
    \mathrm{II} = \re\left(F\frac{G'}{G^2}\rmd z^2\right),
\mbox{ and }
    K = - \left\vert\frac{ G'}{F G^2}\right\vert^2,
\end{equation}
respectively.
\end{theorem}
\begin{proof}[of Lemma \ref{lem::1st2ndFFandKminComplexCurve}]
We will do the proof for $\phi_2=\rmi\phi_1$, the proof for $\phi_2=-\rmi\phi_1$ being entirely analogous. The coefficients of the first fundamental form are
$$
    \langle\mathbf{x}_u,\mathbf{x}_u\rangle = \sum_{j=1}^2\frac{\phi_j^2+2\phi_1\bar{\phi}_1+\phi_j^2}{4} = \frac{\vert \phi_1\vert^2+\vert \phi_2\vert^2}{2},
    \langle\mathbf{x}_u,\mathbf{x}_v\rangle = \sum_{j=1}^2\frac{\phi_j^2-\bar{\phi}_j^2}{-4\rmi} = 0,
$$
and
\begin{eqnarray*}
    \langle\mathbf{x}_v,\mathbf{x}_v\rangle & = & -\sum_{j=1}^2\frac{\phi_j^2-2\phi_1\bar{\phi}_1+\phi_j^2}{4} = \frac{\vert \phi_1\vert^2+\vert \phi_2\vert^2}{2}.
\end{eqnarray*}
Substituting $\phi_2=\rmi\phi_1$ gives the desired expression for the metric. 

For the second fundamental form, first note that the minimal normal is given by $\mathbf{N}_m=(-\frac{1}{2}(\frac{\phi_3}{\phi_1}+\frac{\bar{\phi}_3}{\bar{\phi}_1}),-\frac{1}{2\rmi}(\frac{\phi_3}{\phi_1}-\frac{\bar{\phi}_3}{\bar{\phi}_1}),1)$. Then, the coefficients of the second fundamental form are
\begin{eqnarray*}
 h_{11} & = & \frac{\phi'+\bar{\phi}'}{2}\cdot\mathbf{N}_m\\
 & = & (\frac{\phi_1'+\bar{\phi}_1'}{2},\frac{\rmi\phi_1'-\rmi\bar{\phi}_1'}{2},\frac{\phi_3'+\bar{\phi}_3'}{2})\cdot(-\frac{1}{2}(\frac{\phi_3}{\phi_1}+\frac{\bar{\phi}_3}{\bar{\phi}_1}),-\frac{1}{2\rmi}(\frac{\phi_3}{\phi_1}-\frac{\bar{\phi}_3}{\bar{\phi}_1}),1)\\
 & = & \frac{1}{2}\Big(\phi_1(\frac{\phi_3}{\phi_1})'+\overline{\phi_1(\frac{\phi_3}{\phi_1})'}\Big)= \re[\phi_1(\frac{\phi_3}{\phi_1})'],
\end{eqnarray*}
\begin{eqnarray*}
 h_{12} & = & -\frac{\phi'-\bar{\phi}'}{2\rmi}\cdot\mathbf{N}_m\\
 & = & (-\frac{\phi_1'-\bar{\phi}_1'}{2\rmi},-\frac{\rmi\phi_1'+\rmi\bar{\phi}_1'}{2\rmi},-\frac{\phi_3'-\bar{\phi}_3'}{2\rmi})\cdot(-\frac{1}{2}(\frac{\phi_3}{\phi_1}+\frac{\bar{\phi}_3}{\bar{\phi}_1}),-\frac{1}{2\rmi}(\frac{\phi_3}{\phi_1}-\frac{\bar{\phi}_3}{\bar{\phi}_1}),1)\\
 & = & -\frac{1}{2}\Big(\phi_1(\frac{\phi_3}{\phi_1})'-\overline{\phi_1(\frac{\phi_3}{\phi_1})'}\Big)= -\im[\phi_1(\frac{\phi_3}{\phi_1})'],
\end{eqnarray*}
and $h_{22} = -h_{11}=\re[\phi_1(\frac{\phi_3}{\phi_1})']$. Finally, the Gaussian curvature becomes
\begin{equation*}
    K = -\frac{(\re[\phi_1(\frac{\phi_3}{\phi_1})'])^2+(\im[\phi_1(\frac{\phi_3}{\phi_1})'])^2}{\vert \phi_1\vert^4}=%-\frac{\vert \phi_1\vert^2\vert (\frac{\phi_3}{\phi_1})'\vert^2}{\vert \phi_1\vert^4} = 
    -\frac{1}{\vert \phi_1\vert^2}\left\vert (\frac{\phi_3}{\phi_1})'\right\vert^2.
\end{equation*}
\qed
\end{proof}

From the second fundamental form in Theorem \ref{thr::1st2ndFFandKMinNormal} we conclude that $\mathbf{v}=v_1\mathbf{x}_u+ v_2\mathbf{x}_v$ points to an asymptotic direction if and only if $FG'(v_1+\rmi v_2)^2\in \rmi\mathbb{R}$ and to a principal curvature direction if and only if $FG'(v_1+\rmi v_2)^2\in \mathbb{R}$. An analogous conclusion is valid for the data $\phi=(F,-\rmi F,-\frac{F}{G})$.

For every $\theta$, the minimal surface $\mathbf{x}_{\theta}$ associated with the Weierstrass data $({\rme^{-\rmi\theta}}F,G)$ is isometric to the surface $\mathbf{x}_{0}$ associated with the data $(F,G)$ and together they form {an} \emph{associate family} of minimal surfaces: $${\mathbf{x}_{\theta}=\cos(\theta)\,\mathrm{Re}\int\phi\,\rmd z+\sin(\theta)\,\mathrm{Im}\int\phi\,\rmd z}.$$
The surface with $\theta=\frac{\pi}{2}$ is called the \emph{conjugate} of the  surface $\mathbf{x}_{0}$. Therefore, it follows as a corollary of Theorem \ref{thr::1st2ndFFandKMinNormal} that asymptotic {and} principal directions of $\mathbf{x}_{\theta}$ are mapped {respectively} into principal {and} asymptotic directions of its conjugate surface $\mathbf{x}_{\theta+\frac{\pi}{2}}$.

\subsection{Other choices of a Weierstrass representation}
\label{subsect::OtherChoicesWeierstRep}

The holomorphic data in Sato's choice of representation is $\phi=(F,\rmi F,2FG)$ \cite{SatoArXiv2018}. For this representation, we have $\tilde{\mathbf{N}}_m=2G$ and $\pi(\xi)=-\frac{2}{\bar{G}}$. Moreover, in Ref. \cite{SatoArXiv2018} it is also provided the expression for the second fundamental form and Gaussian curvature. From the above, $\mathrm{II}$ and $K$ of the minimal immersion $\mathbf{x}=\re\int(F,\rmi F,2FG)$ can be further simplified and written as 
$$ \mathrm{II} = 2\,\re(FG'\rmd z^2)
\mbox{ and }K = -4\Big\vert\frac{G'}{F}\Big\vert^2.$$

It is worth mentioning that the motivation of Ref. \cite{SatoArXiv2018} was to understand stationary surfaces in 4d Minkowski space $\mathbb{E}_1^4$, i.e., spacelike surfaces with zero mean curvature. Indeed, from the fact that $\mathbb{I}^3$ can be isometrically immersed in $\mathbb{E}_1^4$, it is established that there exists a one-to-one correspondence between flat stationary and simply isotropic minimal surfaces, up to rigid motions in the respective spaces. In addition, using that the 3d Euclidean and Minkowski spaces can also be isometrically immersed in $\mathbb{E}_1^4$, it is shown that minimal, maximal, and simply isotropic minimal surfaces in $\mathbb{E}^3$, $\mathbb{E}_1^3$, and $\mathbb{I}^3$, {respectively,} can be associated with members of the 1-parameter family $\{f_{\theta}\}_{\theta\in[0,2\pi]}$ of stationary surfaces given by \cite{SatoArXiv2018}
\begin{equation}
    f_{\theta}=\re\int^z\Big(F(1-\cos2\theta\,G^2),\rmi F(1+\cos2\theta\,G^2),2\cos\theta\,FG,2\sin\theta\,FG\Big)\rmd z.
\end{equation}
In fact, the immersions $f_{0},f_{\frac{\pi}{4}}$, and $f_{\frac{\pi}{2}}$ can be associated with minimal, simply isotropic minimal, and maximal surfaces in $\mathbb{E}^3$, $\mathbb{I}^3$, and $\mathbb{E}_1^3$, respectively. Moreover, these stationary surfaces are contained in a 3d subspace {of $\mathbb{E}_1^4$} and it is possible to deduce that $f_{0},f_{\frac{\pi}{4}}$, and $f_{\frac{\pi}{2}}$ correspond to surfaces with Gaussian curvature $K\leq0$, $K=0$, and $K\geq0$ \cite{MaAM2013,SatoArXiv2018}, respectively\footnote{The correspondence with minimal surfaces in $\mathbb{I}^3$ is a special feature of $\mathbb{E}_1^4$ since a zero mean curvature surface in 4d Euclidean space with zero Gaussian curvature must be a plane.}.

Finally, the representations we proposed in this work, and also the one proposed by Sato, are by no means the only possible choices. For example, it seems natural to choose $\phi_1=F$, $\phi_2=-\rmi F$, and $\phi_3=G$, but note that $\mathbf{N}_m=\frac{1}{\vert F\vert^2}(\mathrm{Re}(\bar{F}G),\mathrm{Im}(\bar{F}G),1)$, which implies  $\tilde{\mathbf{N}}_m=G/F$ and $\pi\circ\xi=\bar{F}/\bar{G}$. Therefore, neither $F$ nor $G$ have a clear geometric meaning as in (\ref{eq::WeierstrassRepWithGtopviewNm}) or (\ref{eq::WeierstrassRepWithGprojParNor}). Yet another possibility is to choose $\phi_1=FG$, $\phi_2=\rmi FG$, and $\phi_3=F$. Here, the minimal normal is $\mathbf{N}_m=\frac{1}{\vert G\vert^2}(\re\,G,\im\,G,1)$, which implies $\tilde{\mathbf{N}}_m=1/\bar{G}$ and $\pi\circ\xi=G$. The latter has the geometric meaning we seek, but the determinant of the metric is $\det g_{ij}=2\vert F\vert^2\vert G\vert^2$ and, consequently, we can no longer associate the singularities of the minimal immersion with the zeros of a single function. 

\subsection{{Examples}}
\label{subsect::ExamplesWeierstrassRepr}

{We are going to build examples of  minimal surfaces  corresponding to the isotropic counterparts of the helicoid, catenoid, and Scherk surfaces using the holomorphic representation. (See Sato \cite{SatoArXiv2018} for further examples.)}

{\textit{(a) The helicoid and logarithmoid of revolution}: Consider the holomorphic data $F=1$ and $G=-z/p$, where $p\in\mathbb{R}$. Using the minimal normal representation, we have $ \int(F,\rmi F,-FG)\rmd z = (z,\rmi z,\frac{z^2}{2p})+(a,b,c)$. Thus, the corresponding minimal surface is the hyperbolic paraboloid:}
\begin{equation}
    \mathbf{x}(z=u+\rmi v) = (u,-v,\frac{u^2-v^2}{2p}). 
\end{equation}
{The corresponding conjugate surface is the hyperbolic paraboloid}
\begin{equation}
    \mathbf{x}_{\frac{\pi}{2}}(z=u+\rmi v) = (v,u,\frac{uv}{p}). 
\end{equation}

{On the other hand, using the parabolic normal representation,  we have $\int(F,-\rmi F,-F/G)\rmd z = (z,-\rmi z,p\ln z)+(a,b,c)$. Thus, up to translations, the corresponding minimal surface is the logarithmoid of revolution (see Fig. \ref{fig:AssociatedFamilyHelicoid}):}
\begin{equation}
    \mathbf{x}(z=r\rme^{\rmi \varphi}) = (r\cos\varphi,r\sin\varphi,p\ln r).
\end{equation}
{The corresponding conjugate surface is the helicoid (see Fig. \ref{fig:AssociatedFamilyHelicoid}):}
\begin{equation}
    \mathbf{x}_{\frac{\pi}{2}}(z=r\rme^{\rmi \varphi}) = (r\sin\varphi,-r\cos\varphi,p\,\varphi).
\end{equation}

\begin{figure}
    \centering
    \includegraphics[width=\linewidth]{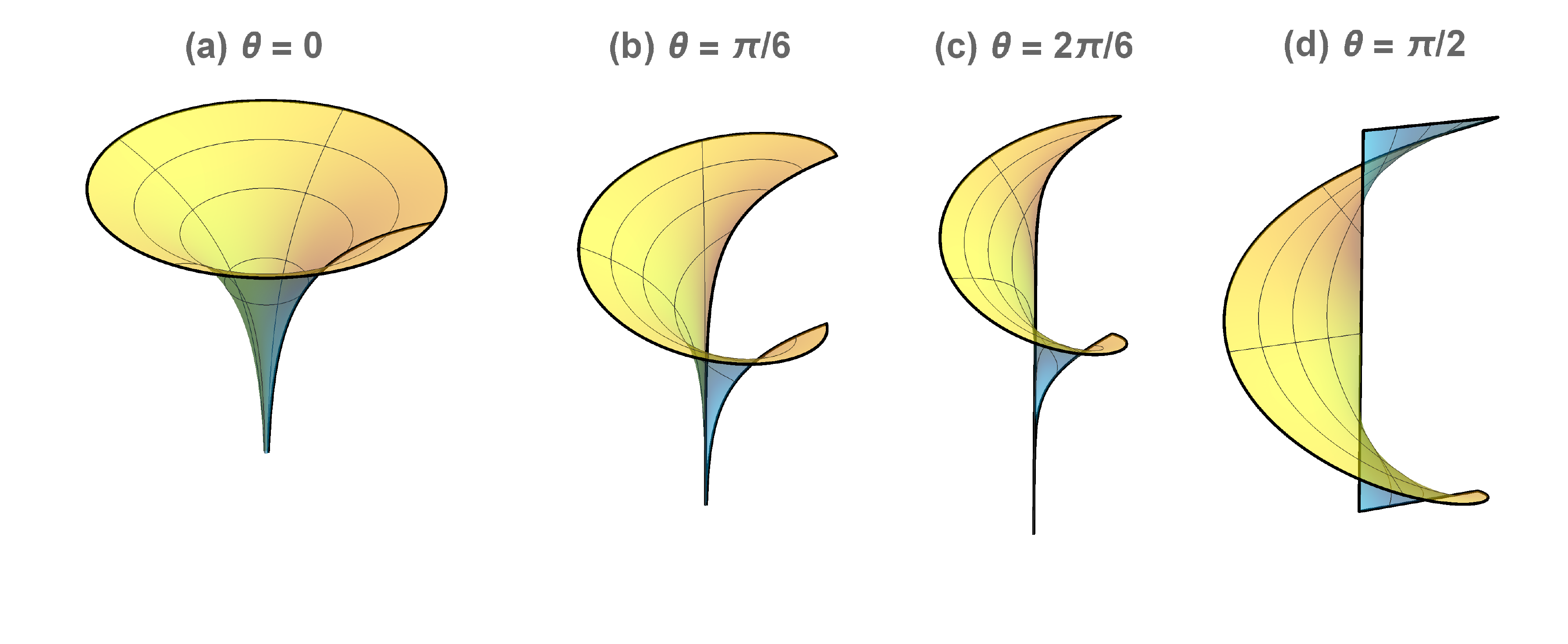}
    \caption{{The associate family of simply isotropic minimal surfaces of the helicoid parametrized as $\mathbf{x}_{\theta}(r,\varphi)=(r\cos(\varphi-\theta),r\sin(\varphi-\theta),p\ln r\cos\theta+p\,\varphi\sin\theta)$, where $p\in\mathbb{R}$ and $\theta\in[0,\frac{\pi}{2}]$. The family $\mathbf{x}_{\theta}$ provides the helicoid for $\theta=\frac{\pi}{2}$, figure (d), and the logarithmoid of revolution for $\theta=0$, figure (a), which is the only minimal surface of revolution in simply isotropic space. (In the figures, the isotropic $z$-direction points in the vertical.)}}
    \label{fig:AssociatedFamilyHelicoid}
\end{figure}

\begin{figure}
    \centering
    \includegraphics[width=\linewidth]{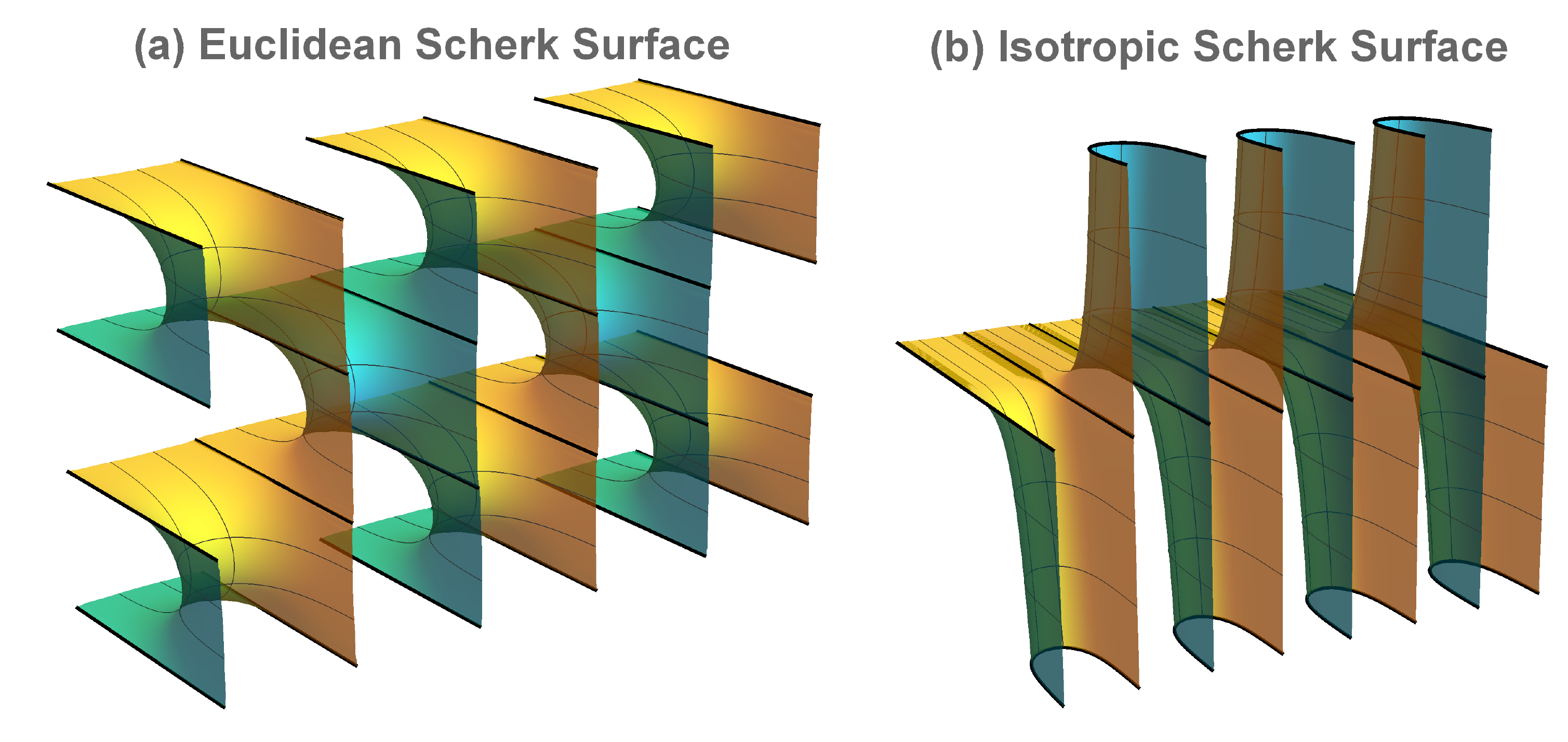}
    \caption{{Minimal Scherk surfaces: (a) The minimal surface in Euclidean space $\mathbb{E}^3$ implicitly defined by $x_1=\ln\frac{\cos x_3}{\cos x_2}$; and (b) The minimal surface in the simply isotropic  space $\mathbb{I}^3$ implicitly defined by $x_1=\ln\frac{x_3}{\cos x_2}$. Note that in $\mathbb{E}^3$ the Scherk surface is doubly periodic, while in $\mathbb{I}^3$ it is only singly periodic. (In the figures, the $z$-direction points in the vertical.)}}
    \label{fig:ScherkSurf}
\end{figure}

{In $\mathbb{E}^3$, the conjugate surface of the helicoid is the catenoid, which is the only minimal surface of revolution in $\mathbb{E}^3$. The logarithmoid of revolution is the only simply isotropic minimal surface of revolution \cite{daSilvaMJOU2021,Sachs1990} and, therefore, we may see it as the counterpart of the catenoid in $\mathbb{I}^3$.}

{\textit{(b) The Scherk surfaces}: The so-called first Scherk surface is the minimal surface obtained by moving a plane curve along another plane curve. The solution in $\mathbb{E}^3$ is provided by the doubly periodic minimal  surface $x_1=\ln\frac{\cos x_3}{\cos x_3}$. We may pose the same problem in $\mathbb{I}^3$. However, in $\mathbb{I}^3$, we must distinguish between three types of planes $\Pi_1$ and $\Pi_2$ containing the generating curves \cite{StrubeckerCrelle1954}. Namely, we have the:}
\begin{enumerate}
    \item Scherk surface of first type: $\Pi_1$ and $\Pi_2$ are both isotropic, e.g., $\Pi_1:x=\gamma\, y$, $\Pi_2:y=-\gamma\, y$ ($\gamma$ constant);
    \item Scherk surface of second type: $\Pi_1$ is isotropic and $\Pi_2$ is not isotropic, e.g., $\Pi_1:y=0$, $\Pi_2:z=0$;
    \item Scherk surface of third type: $\Pi_1$ and $\Pi_2$ are both non isotropic, e.g., $\Pi_1:y-z=\frac{\pi}{2}$, $\Pi_2:y+z=\frac{\pi}{2}$.
\end{enumerate}
{Those surfaces are 1. planes and hyperbolic paraboloid whose cross sectional hyperbolas lie on isotropic planes; 2. the simply periodic surface $x_1=-\ln\frac{x_3}{\cos x_2}$; and 3. the doubly periodic surface $\mathbf{x}(u^1,u^2)=(\ln\frac{\sin u^1}{\sin u^2},u^1+u^2,u^1-u^2)$.}

{We already obtained the isotropic holomorphic representation of hyperbolic paraboloids. Now, consider the holomorphic data $F=\frac{1}{z}$ and $G=-\frac{1}{z}$. Using the parabolic normal representation, we have $\int(F,-\rmi F,-\frac{F}{G})=(\ln z,-\rmi \ln z,z)$. Thus, up to translations, the corresponding minimal surface is}
\begin{equation}
    \mathbf{x}(z)=\re(\ln z,-\rmi \ln z,z)=(\ln\vert z\vert,\arg z,\re\, z).
\end{equation}
{Using the identity $\cos(\arctan x)=\frac{1}{\sqrt{1+x^2}}$, we see that the map $\mathbf{x}=(x_1,x_2,x_3)$ parametrizes the Scherk surface of second type (see Fig. \ref{fig:ScherkSurf}.(b).) $x_1 = \ln \frac{x_3}{\cos x_2}$.}
{Alternatively, defining $w=\ln z$, we reparametrize $\mathbf{x}(z)$ as the harmonic graph}
\begin{equation}
    \mathbf{x}(w=u+\rmi v)=(w,\re(\rme^w))=(u,v,\rme^u\cos v).
\end{equation}

%{The corresponding conjugate surface is given by}
%\begin{equation}
%    \mathbf{x}_{\frac{\pi}{2}}(z) = -\im(\ln z,-\rmi \ln z,z)=(-\arg z,\ln\vert z\vert,-\im\, z).
%\end{equation}
%{As before, defining $w=\ln z$, we may reparametrize $\mathbf{x}(z)=(\rmi\ln z,-\im\, z)\in\mathbb{C}\times\mathbb{R}$ as the harmonic graph}
%\begin{equation}
%    \mathbf{x}_{\frac{\pi}{2}}(w=u+\rmi v)=(\rmi\, w,-\im(\rme^w))=(-v,u,-\rme^u\sin v).
%\end{equation}
%{By performing the translation $v\mapsto v-\frac{\pi}{2}$, we see that $\mathbf{x}_{\frac{\pi}{2}}$ is essentially the same surface as $\mathbf{x}_{0}$ and, therefore,  we see that the simply isotropic Scherk surface of the second type is self-dual. (?????)}

{Finally, concerning  the Scherk surface of third type, Strubecker showed that it can be reparametrized as the harmonic graph of $h(z)=\im\,\ln \cosh z$ \cite{StrubeckerCrelle1954}. Therefore, the corresponding parabolic holomorphic data can be taken as $F=1$ and $G=-\rmi\coth z$. (Despite the huge resemblance with the classic minimal Scherk surface, the graph of $h(z)=\im\ln \cosh z$ is not minimal in $\mathbb{E}^3$.)}

\section{Isotropic Bj\"orling representation}

The Bj\"orling representation consists in the Cauchy problem for the minimal surface equation. More precisely, given an analytic curve $c(s)$ and a unit vector $\mathbf{e}$ along $c$, with $\{c',\mathbf{e}\}$ linearly independent, find the minimal surface $S$ which contains $c$ and such that $\mathbf{e}$ is tangent to $S$ along $c$. In Euclidean space, we may equivalently prescribe a unit vector field $\mathbf{n}$ such that $\mathbf{n}$ is normal to the sought minimal surface along $c$. 

If we prescribe the tangent planes along $c$, $s\mapsto\mbox{span}\{c'(s),\mathbf{e}(s)\}$, then from Eq. (\ref{eq::TangVecGenericHolRep}) we may set $\re(\phi)=c'$ and $\im(\phi)=-\mathbf{e}$ and the corresponding minimal surface is parameterized by
\begin{equation}
    \mathbf{x}(z)=\re\int\Big[c'(z)-\rmi \,\mathbf{e}(z)\Big]\rmd z,
\end{equation}
where $c(z)$ and $\mathbf{e}(z)$ are analytical extensions of $c(s)$ and of $\mathbf{e}(s)$, respectively.

If we prescribe the minimal normal $\mathbf{N}_m$ along $c$, then $\mathbf{N}_m\times c'$ is a tangent vector and we can use the previous construction. Indeed, from Eq. (\ref{eq::TangVecGenericHolRep}) we may set $\re(\phi)=c'$ and $\im(\phi)=-\mathbf{N}_m\times c'$ and the corresponding minimal surface is parameterized by
{
\begin{equation}
    \mathbf{x}(z)=\re\int\Big[c'(z)-\frac{\rmi}{\Omega(z)} \,\mathbf{N}_m(z)\times c'(z)\Big]\rmd z,
\end{equation}
}
where {$\Omega=\sqrt{\langle\mathbf{N}_m\times c',\mathbf{N}_m\times c'\rangle}$ and} $c(z)$ and $\mathbf{N}_m(z)$ are analytical extensions of $c(s)$ and of $\mathbf{N}_m(s)$, respectively. 

Finally, if we prescribe the parabolic normal $\xi$ along $c$, then we can construct a tangent vector given by {$\mathbf{e}=(\xi+\frac{1+\Vert\tilde{\xi}\Vert^2}{2}\mathcal{N})\times c'$}. Now, the corresponding minimal surface is parameterized by
{
\begin{equation}
    \mathbf{x}(z)=\re\int\Big[c'(z)-\frac{\rmi}{\Omega(z)} \,\Big(\xi(z)+\frac{1+\Vert\tilde{\xi}(z)\Vert^2}{2}\mathcal{N}\Big)\times c'(z)\Big]\rmd z,
\end{equation}
}
where {$\Omega=\sqrt{\langle\mathbf{e},\mathbf{e}\rangle}$ and} $c(z)$ and $\xi(z)$ are analytical extensions of $c(s)$ and of $\xi(s)$, respectively.

\subsection{{Examples}}

\begin{figure}
    \centering
    \includegraphics[width=\linewidth]{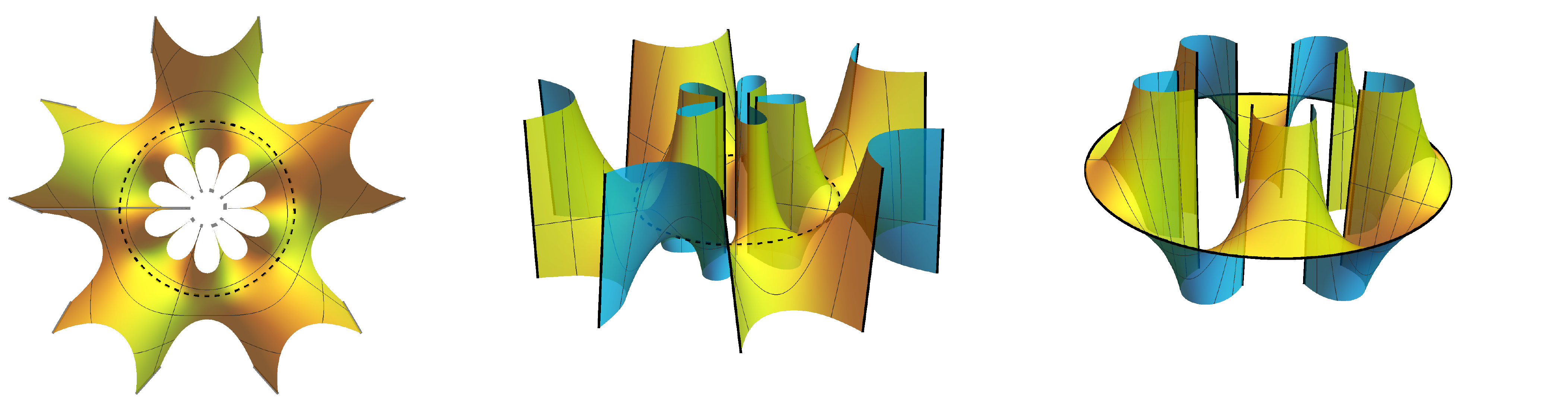}
    \caption{{Bent Scherk surface obtained from the Bj\"orling representation by prescribing an oscillating tangent plane along a circle (dashed black line) instead of an oscillating tangent plane along a line as in the usual Scherk surface (see Figure \ref{fig:ScherkSurf}): (Left) View from the top of the surface; (Center) Isometric projection of the surface; and (Right) Isometric projection of the surface showing only the region in the interior of the initial curve. (In the figures, we set $\lambda=5$ in Eq. (\ref{eq::BentScherk}) and the degenerate $z$-direction points in the vertical.)}}
    \label{fig:BentScherk}
\end{figure}

{We now build examples of simply isotropic minimal surfaces using the Bj\"orling representation corresponding to the isotropic counterparts of the catenoid (logarithmoid of revolution) and Scherk surfaces. In addition, we also build a bent Scherk surface, i.e., instead of a surface composed of a linear chain we obtain a circular one. (See figures \ref{fig:ScherkSurf} and \ref{fig:BentScherk}.)}

{\textit{(a) Logarithmoid of revolution:} In Euclidean space, we may obtain the catenoid by prescribing the analytic curve $c(s)=(\cos s,\sin s,0)$ and the vector field $\mathbf{e}=(0,0,1)$ and demand $\mathbf{e}$ to be tangent to the surface along $c$. However, in the simply isotropic space we should add some inclination to $\mathbf{e}$ to avoid isotropic tangent planes. Thus, consider the vector field $\mathbf{e}(s)=(\cos s,\sin s,p)$ along $c$. Then, the holomorphic curve associated with the Bj\"orling problem is 
$$c'(z)-\rmi \mathbf{e}(z)=(-\sin z-\rmi\cos z,\cos z-\rmi\sin z,-\rmi p)=(-\rmi\rme^{-\rmi  z},\rme^{-\rmi z},-\rmi p),$$
whose integration gives $\int(c'-\rmi\mathbf{e})\rmd z=(\rme^{-\rmi z},\rmi \rme^{-\rmi z},-\rmi pz)$. Now, performing the coordinate change $w=\rme^{-\rmi z}$, we finally obtain
\begin{equation}
    \mathbf{x}(w=r\rme^{-\rmi\varphi})=\re(w,\rmi w,p\ln w)=(r\cos\varphi,r\sin\varphi,p\ln r),
\end{equation}
which parametrizes the logarithmoid of revolution.}

{\textit{(b) Scherk surface:} Let us consider the analytic curve $c(s)=(0,s,A\cos s)$ and the vector field $\mathbf{e}=(1,0,A\cos s)$, where $A\in\mathbb{R}$. Then, the holomorphic curve associated with the Bj\"orling problem is 
$$c'(z)-\rmi \mathbf{e}(z)=(-\rmi,1,-A\sin z-\rmi A\cos z)=(-\rmi,1,-\rmi A\,\rme^{-\rmi z}),$$
whose integration gives $\int(c'-\rmi\mathbf{e})\rmd z=(-\rmi z,z,A\, \rme^{-\rmi z})$. We finally obtain 
\begin{equation}
    \mathbf{x}(z=v+\rmi u)=(u,v,A\,\rme^u \cos v),
\end{equation}
which parametrizes the Scherk surface of the second type (see Subsection \ref{subsect::ExamplesWeierstrassRepr}).}

{\textit{(c) Bent Scherk surface:} Instead of a singly periodic minimal surface, let us build a Scherk surface with a bent core. Consider the analytic curve $c(s)=(\sin s,\cos s,0)$ and the vector field $\mathbf{e}=(\sin s,\cos s,A\lambda\cos \lambda s)$, where $\lambda,A\in\mathbb{R}$. Then, the holomorphic curve associated with the Bj\"orling problem is 
$$c'-\rmi \mathbf{e}=(\cos z-\rmi \sin z,-\sin z-\rmi \cos z,-\rmi A\lambda \cos\lambda z)=(\rme^{-\rmi z},-\rmi \rme^{-\rmi z},-\rmi A\lambda \cos\lambda z),$$
whose integration gives $\int(c'-\rmi\mathbf{e})\rmd z=(\rmi\rme^{-\rmi z},\rme^{-\rmi z},-\rmi A\sin\lambda z)$. Performing the coordinate change $w=\rmi\,\rme^{-\rmi z}$ [$z=\rmi \ln(-\rmi w)$], we finally obtain the bent Scherk surface (see figure \ref{fig:BentScherk})
\begin{eqnarray}
    \mathbf{x}(w=r\rme^{\rmi \varphi}) &=& \re(w,-\rmi w,-\rmi\,A\sin[\rmi\lambda\ln(-\rmi w)])\nonumber\\
    %&=& \re(w,-\rmi w,-\rmi\,A\sin[\rmi\lambda\ln r\rme^{\rmi (\varphi-\frac{\pi}{2})}])\nonumber\\
    & = & \Big(r\cos\varphi,r\sin\varphi,A\cos[\lambda(\frac{\pi}{2}-\varphi)]\sinh(\lambda\ln r)\Big).\label{eq::BentScherk}
\end{eqnarray}
To the best of our knowledge, the construction of a bent Scherk surface has never been reported on the literature. (A similar procedure may be used to construct a bent helicoid as done in $\mathbb{E}^3$. See Ref. \cite{LopezMMJ2018} and references therein.)}

\section{Concluding remarks}

In this work, we pushed further  results in Ref. \cite{SatoArXiv2018} concerning the Weierstrass representation of minimal surfaces in simply isotropic space $\mathbb{I}^3$ by providing a way to associate part of the holomorphic data with the choice of a Gauss map. We also discussed simply isotropic analogs of the Bj\"orling representation, which correspond to the Cauchy problem for the minimal surface equation.

Sato's choice for the Weierstrass representation in Ref. \cite{SatoArXiv2018} was motivated by the study of stationary surfaces in 4d Minkowski space $\mathbb{E}_1^4$, i.e., zero mean curvature spacelike surfaces. In fact, simply isotropic minimal surfaces are put in correspondence with flat stationary surfaces, while minimal and maximal surfaces in the 3d Euclidean and Minkowski spaces are put in correspondence with  stationary surfaces of curvature $K\leq0$ and $K\geq0$, respectively. 

Minimal, simply isotropic minimal, and maximal surfaces are associated with members of a 1-parameter family $\{f_{\theta}\}$ of stationary surfaces. The family of immersions $f_{\theta}$ though does not exhaust the class of stationary surfaces in $\mathbb{E}_1^4$. Indeed, from Theorem 2.4 of Ref. \cite{MaAM2013}, given two holomorphic functions $\phi,\psi$ and a holomorphic 1-form $\rmd h$ satisfying a set of regularity conditions one can generically represent stationary surfaces in $\mathbb{E}_1^4$ as
\begin{equation*}
\mathbf{x}(z)=2\,\re\int^z(\phi+\psi,-\rmi(\phi+\psi),1-\phi\psi,1+\phi\psi)\rmd h.    
\end{equation*}
The functions $\phi,\psi$ are associated with the Gauss maps of the stationary surface while $\rmd h$ is the height differential. 
Minimal, maximal, and simply isotropic minimal surfaces in $\mathbb{E}^3$, $\mathbb{E}_1^3$, and $\mathbb{I}^3$ respectively correspond to considering $\phi\equiv-1/\psi$, $\phi\equiv1/\psi$, and $\psi\equiv0$ (or $\phi\equiv0$) in the representation above. 

This later observation poses the interesting problem of finding other geometries that can be isometrically immersed in $\mathbb{E}_1^4$ and whose corresponding zero mean curvature surfaces can be put in correspondence with stationary surfaces of  $\mathbb{E}_1^4$. In addition, we may also ask how many non-equivalent geometries are needed in order to exhaust all stationary surfaces in $\mathbb{E}_1^4$.

In conclusion, the take-home message of this work is that there are multiple ways of holomorphically represent simply isotropic minimal surfaces. However, distinct choices have distinct advantages/disadvantages and, therefore, when deciding for a representation or another we should have in mind which specific geometric feature the chosen representation will allow us to control.

\begin{acknowledgement}
We would like to thank useful discussions with Alev Kelleci Akbay (Firat University) and Yuichiro Sato (Tokyo Metropolitan University). This work has been financially supported by the Mor\'a Miriam Rozen Gerber scholarship for Brazilian postdocs.
\newline
\textbf{Conflict of interest} The author declares he has no conflict of interest.
\end{acknowledgement}


\end{document}